\documentclass[reqno, 12pt]{amsart}
\usepackage{a4wide, amsfonts, amssymb, amsmath, amsthm, mathrsfs, enumerate, enumitem, setspace, url}
\usepackage[utf8]{inputenc}
\usepackage{graphicx}
\usepackage{tikz-cd} 
	\usetikzlibrary{cd}
	\usetikzlibrary{graphs}
\usepackage[all]{xy}
\usepackage{mathtools}
\usepackage{float}
\usepackage{longtable}
\usepackage{placeins}
\usepackage{pifont}
\usepackage{color}
\usepackage{algpseudocode}
\usepackage{hyperref} 
	\definecolor{darkred}{rgb}{0.5,0,0}
	\definecolor{darkgreen}{rgb}{0,0.5,0}
	\definecolor{darkblue}{rgb}{0,0,0.5}
	\hypersetup{linkcolor=darkblue,filecolor=darkgreen,urlcolor=darkred,citecolor=darkblue}

\usepackage{comment}
\usepackage[T2A,T1]{fontenc}
\DeclareSymbolFont{cyrillic}{T2A}{cmr}{m}{n}
\DeclareMathSymbol{\Sha}{\mathalpha}{cyrillic}{216}

\usepackage[toc,page]{appendix}

\theoremstyle{plain}
\newtheorem{theorem}{Theorem}[section]
\newtheorem*{theorem*}{Theorem}
\newtheorem{proposition}[theorem]{Proposition}
\newtheorem{lemma}[theorem]{Lemma}
\newtheorem{conjecture}[theorem]{Conjecture}
\newtheorem{corollary}[theorem]{Corollary}
\theoremstyle{definition}
\newtheorem{algorithm}[theorem]{Algorithm}
\theoremstyle{remark}
\newtheorem{remark}[theorem]{Remark}

\newtheorem*{listofremainingcase}{List of remaining cases}

\newtheorem*{acknowledgements}{Acknowledgements}
\newtheorem*{key}{Key for tables}

\theoremstyle{definition}
\newtheorem{definition}[theorem]{Definition}
\newtheorem{notation}[theorem]{Notation}
\numberwithin{equation}{section}

\newcommand{\eln}{l^{n}}

\DeclareMathOperator{\HH}{H}
\DeclareMathOperator{\Gal}{Gal}
\DeclareMathOperator{\Pic}{Pic}
\DeclareMathOperator{\Proj}{Proj}
\DeclareMathOperator{\Br}{Br}

\DeclareMathOperator{\Spec}{Spec}

\DeclareMathOperator{\Aut}{Aut}
\DeclareMathOperator{\tr}{Tr}
\DeclareMathOperator{\id}{Id}

\DeclareMathOperator{\Frob}{Frob}

\DeclareMathOperator{\rank}{rank}
\DeclareMathOperator{\charr}{char}

\DeclareMathOperator{\et}{\acute{e}t}
\renewcommand{\epsilon}{\varepsilon}

\begin{document}
\onehalfspacing
\title[Singular del Pezzo surfaces over finite fields]{Singular del Pezzo surfaces over finite fields}

\author{H. Uppal}
	\address{H. Uppal, Department of Mathematical Sciences, University of Bath, Claverton Down, Bath, BA2 7AY, UK}
	\email{hsu20@bath.ac.uk}

\date{\today}

\begin{abstract}
If $X$ is a singular del Pezzo surface of degree $d$ over a finite field $\mathbb{F}_{q}$ with only rational double point singularities, does there always exist a smooth $\mathbb{F}_{q}$-point on $X$? We show that this is true for $d\geq 3$ and give counterexamples in the case of $d=2$. 
\end{abstract}
\maketitle
\setcounter{tocdepth}{1}
\tableofcontents

\section{Introduction}\label{sec:intro}
We say a field $k$ is $C_{1}$ if every hypersurface of degree $\leq n$ embedded in $\mathbb{P}^{n}_{k}$ has a $k$-rational point. In an unpublished paper Lang conjectured the following.
\begin{conjecture}
Every smooth proper separably rationally connected variety over a $C_{1}$ field has a rational point.
\end{conjecture}
The simplest type of $C_{1}$ field is a finite field. In \cite{E03} Esnault provides a profound result which shows every smooth Fano variety over a finite field has a rational point. In this paper we consider an analogue of Lang's conjecture for mildly singular Fano varieties, where the most natural question is to ask for the existence of a smooth rational point. One would guess that over an infinite $C_{1}$ field, the set of rational points will dense, hence the interesting case would be finite fields. We specifically consider Fano surfaces over finite fields where the mild singularities are rational double point singularities (Definition \ref{defn: rational ingularities}), these are singular del Pezzo surfaces. There has already been results in this direction, for example a result of Kollár \cite[Thm 2]{K02} shows that every nonconical cubic hypersurface of dimension $n$ in $\mathbb{P}^{n+1}$ over a finite field has a smooth rational point (See Remark \ref{rem: C-W cubic hypersurface}). Building on these results, in this paper we obtain the following.
 
\begin{theorem}\label{thm: main thm} Let $X$ be a singular del Pezzo surface of degree $d$ over a finite field $k$ of size $q$. Then $X$ has a smooth rational point if
\begin{enumerate}
\item $d\geq 3$,
\item $d=2$ and $q\neq 2,4$.
\end{enumerate}
\end{theorem}
Theorem \ref{thm: main thm} extends the work of Kollár to higher degree singular del Pezzo surfaces and also singular del Pezzo surfaces of degree 2 away from finite fields of size 2 and 4. Moreover, using work of Coray and Tsfasman \cite{CT88} we obtain the following corollary to Theorem \ref{thm: main thm}.

\begin{corollary}\label{cor: Main corollary}
Let $X$ be a singular del Pezzo surface over of degree $d\geq 4$, over a finite field $k$. If $X$ is not a Iskovskih surface (see \cite[pg. 74]{CT88}) then $X$ is $k$-rational.
\end{corollary}
The following Theorem shows that the condition $q\neq 2$, was necessary in Theorem \ref{thm: main thm} for singular del Pezzo surfaces of degree 2.
\begin{theorem}\label{thm: main thm 2} There exists a singular del Pezzo of degree 2 of singularity type $A_{1}, 3A_{1}$ and $D_{4}$ over $\mathbb{F}_{2}$ without a smooth rational point.
\end{theorem}
Theorem \ref{thm: main thm 2} shows that this analogue of Lang's conjecture does not hold in general. However, it would interesting to classify which varieties over $C_{1}$ fields have a smooth rational point. This would useful in determining (uni)rationality of singular Fano varieties and also determining $p$-adic solubility of smooth varieties with bad reduction via Hensel's Lemma.
\begin{acknowledgements}
The author would like to thank Daniel Loughran for suggesting the problem and his tireless support. He also like to thank Nathan Kaplan, Martin Bright, Julian Demeio and Jesse Pajwani for useful discussions.
\end{acknowledgements}
\subsection*{Notation}
By a variety over field $k$ we mean a separated integral scheme of finite type over $k$. Throughout the field $k$ will be assumed to be perfect and we will denote by $\bar{k}$ the algebraic closure of $k$.  We reserve $K$ for an arbitrary field. We denote by $\mathbb{F}_{q}$ the finite field of size $q$. Given a scheme $S$ over a field $k$ we denote by $\bar{S}$ for the base change $S:=S\times_{\Spec k}\Spec \bar{k}$.
\subsection*{Outline of paper and methodologies} In Section \ref{sec: general theory} we develop some general theory around group actions on blow ups. Moreover, we give a rigorous definition of singular del Pezzo surfaces and their relation to weak del Pezzo surfaces. In Section \ref{subsec: point count over finite fields} we describe how to count points on singular del Pezzo surfaces over finite fields, by using a Theorem of Kaplan (See Theorem \ref{theorem: Nathan Kaplan result}). To fully take advantage of Theorem \ref{theorem: Nathan Kaplan result} we use computation in SageMath and Magma, which is described in algorithm \ref{algorithm: determine traces on singular del pezzo}. In Section \ref{sec: toric} we show that all toric singular del Pezzo surfaces over a finite field have a smooth rational point. Section \ref{sec: del Pezzo of degree 2} we consider the case of singular del Pezzo surfaces of degree 2. When these surfaces have only one singular rational point in odd characteristic, they have a conic bundle (See Subsection \ref{subsec: conic bundles}). Furthermore, we prove Theorem \ref{thm: main thm} and Corollary \ref{cor: Main corollary} in Subsection \ref{subsec: proof of Main Theorem and Corollary }. We then prove Theorem \ref{thm: main thm 2}  in Subsection \ref{subsec: counterexamples}.

\section{General theory}\label{sec: general theory}
In this section we give a rigorous definition of singular and weak del Pezzo surfaces; to do this, we first describe how group actions change under blow ups.

\subsection{Group actions on schemes}\label{subsec: blow ups}

\begin{proposition}\label{prop: group action on blow up}
Let $X$ be a scheme, with an action by a group scheme $G$ and $Z\subset X$ a closed subscheme of $X$ which is invariant under $G$. Consider the blow up of $X$ along $Z$, denoted by $\pi:X'\rightarrow X$. Then the action of $G$ extends to $X'$ and $\pi^{-1}(Z)$ is also invariant under the action of $G$.
\end{proposition}
\begin{proof}
As $Z$ is invariant under the action of $G$, the inverse image of $Z$ under the morphism $G\times X\rightarrow X$ is $G\times Z$. Hence, we have a morphism $h:G\times X\rightarrow X'$, by the universal property of blow ups \cite[Chap. 2, Prop 7.14]{H77}. Consider the morphism $\psi: G\times X'\rightarrow X'$, defined via the composition \[
G\times X' \xrightarrow{\id \times \pi} G\times X \xrightarrow{h} X'.
\] Now it is sufficient to show that $\psi$, defines a group scheme action of $G$ on $X'$. Let $m$ (resp. $\id_{X}$) be the group law morphism on $X$ (resp. identity morphism) on $X$. Moreover, let $e:\Spec \mathbb{Z} \rightarrow G$ be the identity section of $G$ and $\id_{G}$ the identity morphism on $G$. Then as $\psi$ satisfies \[
\psi\circ \left(\id_{G} \times \psi\right)=\psi \circ (m\times \id_{X}),\ \psi \circ(e\times \id_{X})=\id_{X}
\] on $X'\backslash E \cong X$ where $E$ is the exceptional divisor of the blow up $\pi$, these identities are satisfied on the whole of $X'$ i.e.\ $\psi$ defines a group scheme action of $G$ on $X'$. Then it follows from definition of $\psi$ that $\pi^{-1}(Z)$ is invariant under the action of $G$.
\end{proof}
\begin{definition}
Let $X$ be a normal projective surface $X$ and $n\in\mathbb{Z}$. A \emph{$(-n)$-curve} on $X$ is a
smooth geometrically integral curve $E\subset X$ of genus 0 such that $E^{2} = -n$.
\end{definition}

The following statement is taken from Daniel Loughran's PhD thesis. However, as this thesis is not readily available online we recite the statement and proof here
\begin{lemma}\label{lemma: dan's lemma}
Let $X$ be a smooth projective surface over a field $k$. Then any negative curve on $X$ is the unique effective curve in its divisor class. In particular, if $E$ is a negative curve on $\bar{X}$ whose divisor class is invariant under the action of $\Gal(\bar{k}/k)$, then $E$ is in fact defined over $k$.
\end{lemma} 
\begin{proof}
Suppose $E$ is a negative curve on $X$ and $E'$ is an effective
divisor in $\Pic X$ linearly equivalent but not equal to $E$. This implies $E\not\subset E'$. Hence, $E' \cdot E \geq 0$. However, $E\cdot E' = E^2 <0$ as $E$ is a negative curve, giving a contradiction. To prove the second part of the Lemma, it sufficient to note that the action of $\Gal(\bar{k}/k)$ sends effective divisors to effective
divisors.
\end{proof}

\subsection{Singular del Pezzo surfaces}\label{sec: singular del pezzo}
We now move onto main objects of study, which are singular del Pezzo surfaces. 
\begin{definition}\label{defn: rational ingularities} We say a variety $X$ over a field $k$ has \emph{rational double point singularities} if there exists a resolution of singularities $\pi:Y\rightarrow X$, such that $R^{i}\pi_{*}\mathcal{O}_{Y}=0$ for $i>0$.
\end{definition}
\begin{definition}\label{defn: singular del Pezzo} A \emph{singular del Pezzo surface} is a normal projective surface $X$ with only rational double point singularities, whose anticanonical divisor $-K_{X}$ is ample. Its degree is defined to be $K_{X}^{2}$.
\end{definition}

\subsection{Weak del Pezzo surfaces}\label{sec: weal del pezzos}
We proceed to describing weak del Pezzo surfaces. These surfaces are intrinsically connected to our primary subject of investigation, which are singular del Pezzo surfaces. The geometry of weak del Pezzo surfaces is considerably more manageable in comparison to that of singular del Pezzo surfaces. A significant portion of our examination in this paper will focus on weak del Pezzo surfaces and this approach will enable us to derive conclusions about singular del Pezzo surfaces.

\begin{definition}
A \emph{weak del Pezzo surface} is a smooth surface $X'$ with $-K_{X'}$ nef and big. Its degree is defined to be $K_{X'}^{2}$.
\end{definition}
\begin{theorem}[{\cite[Prop 0.4]{CT88}}] Let $X'$ be a weak del Pezzo surface of degree $d$ over an algebraically closed field $K$. Then $1\leq d\leq 9$ and either \begin{enumerate}
\item $X \cong \mathbb{P}^{1}_{k}\times \mathbb{P}^{1}_{k}$ or $X\cong \mathbb{F}_{2}$ (the Hirzebruch surface),
\item $X$ is the blow up of $\mathbb{P}^{2}_{k}$ in $9-d$ in almost general position.
\end{enumerate} Conversely, every weak del Pezzo surface arises this way.
\end{theorem}
\begin{definition} Let $S$ be a normal projective variety. We define the \emph{anticanonical ring} of $S$ to be the graded ring \[
R(S,-K_{S}):=\bigoplus_{m\geq 0} \HH^{0}(S,-K_{S}).
\] If $R(S,-K_{S})$ is finitely generated then we define the \emph{anticanonical model} of $S$ to be the scheme $\Proj R(S,-K_{S})$.
\end{definition}
\begin{remark}
When $X'$ is a weak del Pezzo the anticanonical model of $X'$ is a singular del Pezzo surface $X$ \cite[Prop 8.3.3]{D12}. Moreover, the minimal desingularisation of a singular del Pezzo surface $X$ is a weak del Pezzo surface $X'$ \cite[Thm 8.1.15]{D12}.
\end{remark}
\begin{proposition}\label{prop: point lying away from -2 curve}
Let $X$ be a singular del Pezzo surface over a field $k$ and $\pi:X'\rightarrow X$ the minimal desingularisation of $X$. Suppose there exists $p\in X'(k)$ not lying on a $(-2)$-curve, then $X$ has a smooth rational point.
\end{proposition}
\begin{proof}
Denote by $\text{Sing}(X)$ the set of singular points on $X$, then $\pi^{-1}(\text{Sing}(X))$ is the set of $(-2)$-curves on $X'$. As $\pi:X'\backslash \pi^{-1}(\text{Sing}(X)) \rightarrow X \backslash \text{Sing}(X)$ is an isomorphism and $p$ lies away from $\pi^{-1}(\text{Sing}(X))$, we have that $X$ has a smooth rational point.
\end{proof}
\begin{remark}
A particular case of Proposition \ref{prop: point lying away from -2 curve}, is when we have a $(-1)$-curve on $\bar{X}'$, which is fixed by the action of $\Gal(\bar{k}/k)$ with a rational point lying away from any $(-2)$-curve. Using Lemma \ref{lemma: dan's lemma}, we can deduce there is a rational point on $X'$ lying away from any $(-2)$-curve.
\end{remark}
\subsection{Dynkin diagrams}
The singularities that we are interested in are rational double point singularities. The resolution graph of these singularities have one vertex for each exceptional curve on the minimal desingularisation above the singularity, and an edge joining two vertices if and only if the corresponding curves intersect. Moreover, each singular point gives rise to a connected component of the resolution graph. The type of each singularity is then defined to be the Dynkin diagram given by its respective connected component in the resolution graph.
\begin{proposition}[{\cite[\S IV Thm 1, \S V Prop 1,\S V Thms 1,2]{DPT80}}]\label{prop: resolution of RDP surfaces and Dynkin diagrams} Let $X$ be a normal surface over a field $k$. Then $x\in X$ is a rational double point singularity if and only if the connected components in the dual graph of the minimal resolution of the point $x$ on $X$ are Dynkin diagrams of type $A_{n}, D_{n}, E_{6}, E_{7}$ or $E_{8}$. Moreover, all irreducible components of the exceptional locus are isomorphic to $\mathbb{P}^{1}$ and have self-intersection $-2$.
\end{proposition}

\begin{remark}
Let $X$ be a singular surface with only rational double point singularities over a field $k$. Then the action of $\Gal(\bar{k}/k)$ on $\bar{X}$ preserves singularity type of points i.e.\ if two singular points $x,y\in \bar{X}(\bar{k})$ have different singularity types, then there does not exist $\sigma \in \Gal(\bar{k}/k)$ such that $\sigma(x)=y$.
\end{remark}

\subsection{Roots systems} We recall some facts about roots systems from \cite[\S 8.2]{D12}. Let $I^{1,N}=\mathbb{Z}^{N+1}$ equipped with the symmetric bilinear form defined by the
diagonal matrix $\text{diag}(1, -1, . . . , -1)$ with respect to the standard basis \[e_{0} = (1,0,...,0),\ e_{1} = (0,1,0,...,0),\dotsc,e_{N} = (0,...,0,1)\]
of $\mathbb{Z}^{N+1}$ . Any basis defining the same matrix will be called an \emph{orthonormal basis}. Consider the vector \[k_{N}=-3e_{0}+\sum_{i=1}^{N}e_{i}\in I^{1,N}.\] We define the \emph{$E_{N}$-lattice} as the sublattice of $I^{1,N}$, defined by \[E_{N}:=(\mathbb{Z}k_{N})^{\perp}.\]
\begin{definition}
A vector $\alpha\in E_{N}$ is called a \emph{root} if $\alpha^{2}=-2$.
\end{definition}
\begin{definition}
A vector $v\in I^{1,N}$ is called \emph{exceptional} if $k_{N}\cdot v = -1$ and $v^{2}=-1$.
\end{definition}
\begin{proposition}[{\cite[Prop 0.4]{CT88}}]\label{prop: picard group of wdP} Let $X'$ be a weak del Pezzo of degree $d\leq 6$, over a field $k$. There is an isomorphism $\Pic \bar{X}^{'} \rightarrow I^{1,9-d}$, i.e.\ there exists a triple $(\Pic X', K_{X'},\langle -, - \rangle)$ where\begin{enumerate}
\item $\Pic \bar{X'}\cong \mathbb{Z}^{10-d}$,
\item $\Pic \bar{X'}$ has a basis basis $l_{0},\dotsc,l_{9-d}$, such that \[
l_{0}^2=1,\ l_{i}^2=-1\ \text{for}\  i\in [1,9-d],\ \langle l_{i}, l_{j}\rangle =0\ \text{for all}\ i\neq j.\] 
\item There is an isomorphism of root lattices $(\mathbb{Z}K_{X'})^{\perp}\cong E_{N}$.
\end{enumerate}
\end{proposition}
\begin{remark}
Under the isomorphism $\Pic X'\rightarrow \mathbb{Z}^{10-d}$ given in Proposition \ref{prop: picard group of wdP}, the $(-2)$-curves on a weak del Pezzo will map to roots and $(-1)$-curves to exceptional vectors.
\end{remark}
\subsection{Graph of negative curves}\label{sec: graph of negative curves}
We discuss how one can determine the graph of negative curves on a weak del Pezzo surface. Throughout Subsection \ref{sec: graph of negative curves}, denote by $X'$ a weak del Pezzo over an algebraically closed field $k$.
\begin{proposition}[{\cite[Thm III.2 and Corollary]{DPT80}}]Denote by $\mathcal{R}$ the subset of $\Pic X'$ containing the $(-2)$-curves of $X'$. An exceptional class $\lambda \in \Pic X'$ is an irreducible effective divisor if and only if $\langle \lambda,\alpha \rangle \geq 0$ for all $\alpha \in \mathcal{R}$.
\end{proposition}
Consider the triple $(\Pic X',' K_{X'},\langle -, - \rangle)$ from Proposition \ref{prop: picard group of wdP}, then $\Pic X'\cong \mathbb{Z}^{N+1}$ and has a basis $l_{0},\dotsc,l_{N}$, such that \[
l_{0}^2=1,\ l_{i}^2=-1\ \text{for}\  i\in [1,N],\ \langle l_{i}, l_{j}\rangle =0\ \text{for all}\ i\neq j.\] 
Denote by $Q$ the space $Q:=\{\alpha \in\Pic X: \langle \alpha, K_{X'}\rangle=0\}$ i.e.\ the orthogonal space to the canonical divisor of $X'$ and as a root lattice we have an isomorphism $Q\cong E_{N}$. The set of roots of $\Pic X'$ is the set $R = \{\alpha \in Q: \alpha^2=-2\}$.  Dolgachev has classified all possible exceptional and root vectors in \cite[Prop 8.2.19, Prop 8.3.7]{D12}. 

\section{Point counts over finite fields}\label{subsec: point count over finite fields}
In this section we describe how weak del Pezzo surfaces can be used to study the number of rational points on singular del Pezzo surfaces.

\begin{theorem}[Weil, {\cite[Chap IV, Thm 27.1]{M74}}]\label{theorem: Weil result}
Let $S$ be a smooth projective surface over a finite field $\mathbb{F}_{q}$. If $\bar{S}$ is rational then 
\[
\#S(\mathbb{F}_{q})=q^2+\tr(\phi^{*})q+1
\] where $\phi$ is the Frobenius endomorphism and $\tr(\phi^{*})$ is the trace of the corresponding representation to $\phi$.
\end{theorem}
\begin{remark}\label{rem: smooth del Pezzo has smooth point}
 Note smooth del Pezzo surfaces are encapsulated in Definition \ref{defn: singular del Pezzo}, however Theorem \ref{theorem: Weil result} shows they always have a smooth $\mathbb{F}_{q}$-point. Hence, we can ignore these surfaces from now on.
\end{remark}
\begin{theorem}[{\cite[Prop 24]{K13}}]\label{theorem: Nathan Kaplan result}
Let $X$ be a singular del Pezzo of degree $9-N\leq 6$ with minimal desingularisation $\pi:X'\rightarrow X$. Let $\mathcal{R} \subset \Pic \bar{X}'$ be the root sublattice generated by $(-2)$-curves on $\bar{X}'$. Then $\#X(\mathbb{F}_{q})=q^{2}+q+1+qt$
where
\[
t = \tr(\phi|_{E_{N}}) - \tr(\phi|_{\mathcal{R}}).
\]
\end{theorem}
\begin{remark}\label{rem: singular del Pezzo with no singular point over ground field has a smooth point}
If $X$ is a singular del Pezzo with no singular rational points (i.e.\ all singular points are contained in a  Galois orbit of degree greater than 1) then by Theorem \ref{theorem: Nathan Kaplan result} we see that $X(\mathbb{F}_{q})\neq \emptyset$, hence $X$ has a smooth rational point.
\end{remark}

\begin{corollary}\label{cor: number of singular points not congruent to one mod q then have smooth point}
Let $X$ be a singular del Pezzo over a finite field $\mathbb{F}_{q}$ with $\delta>0$ singular rational points. If $\delta \not\equiv 1 $ mod $q$, then $X$ has a smooth $\mathbb{F}_{q}$-point. In particular, if $\delta = 2$ then $X$ has a smooth rational point.
\end{corollary}
\begin{proof}
By Theorem \ref{theorem: Nathan Kaplan result} we have $\#X(\mathbb{F}_{q})\equiv 1$ mod $q$. As $X$ has $\delta$ singularities defined over $\mathbb{F}_{q}$, we can deduce that $\#X(\mathbb{F}_{q})\geq \delta$. As $\delta\not\equiv 1$ mod $q$ the number of rational points on $X$ has a lower bound of $\#X(\mathbb{F}_{q})\geq 1+q$. If $q+1>\delta$ then we clearly have a smooth $\mathbb{F}_{q}$-point on $X$. From now on we consider the case $q+1<\delta $. There exists a natural number $N\geq 1$ such that $1+Nq<\delta<1+(N+1)q$. Then $\#X(\mathbb{F}_{q})\geq \delta >1+Nq$ hence, $\#X(\mathbb{F}_{q})\geq 1+(N+1)q$ as $\#X(\mathbb{F}_{q})\equiv 1$ mod $q$.
\end{proof}

\begin{algorithm}\label{algorithm: determine traces on singular del pezzo}
We now detail an algorithm to determine the number of rational points on a singular del Pezzo surface $X$ over $\mathbb{F}_{q}$ of degree $d$, where $d\leq 6$. Denote by $X'$ the minimal desingularisation of $X$ and $\Gamma$ the graph of negative curves on $\bar{X}'$. Note that the action of $\Gal(\bar{\mathbb{F}}_{q}/\mathbb{F}_{q})$ on $\bar{X}'$ will factor through a finite group $\Gal(\mathbb{F}_{q^{n}}/\mathbb{F}_{q})=\langle \Frob_{q}\rangle$, for some $n\in \mathbb{N}$ and define a graph automorphism of $\Gamma$ i.e.\ the action of $\Frob_{q}$ on $\bar{X}'$ corresponds to an element $g\in \Aut(\Gamma)$. Throughout the algorithm $\phi_{g}$ will denote the homomorphism 
\[
\Gal(\mathbb{F}_{q^{n}}/\mathbb{F}_{q})\rightarrow \Aut(\Gamma),\  \Frob_{q}\mapsto g.
\] Let each vertex $v_{i}$ in $\Gamma$ correspond to a basis element $e_{i}$ of the free $\mathbb{Z}$-module $M=\oplus_{i=1}^{n} \mathbb{Z}e_{i}$ with intersection pairing $\langle - , - \rangle$ defined by: \begin{itemize} 
\item $\langle e_{i},e_{i}\rangle=-n$ if $v_{i}$ corresponds to a $(-n)$-curve,
\item $\langle e_{i},e_{j}\rangle=$ number of edges between $v_{i}$ and $v_{j}$ if $i\neq j$. 
\end{itemize}
The intersection form on $M$ has rank $10-d$, denote by $F$ the kernel of this form. One can define an action of $\Gal(\mathbb{F}_{q^{n}}/\mathbb{F}_{q})$ on $M$ via 
\[
\Frob_{q}\cdot m := \phi_{g}(\Frob_{q})m, 
\] for $m\in M$. Note that we can write a basis $l_{0},\dotsc,l_{N}$, for $\Pic \bar{X}'$, such that, \[
l_{0}^2=1,\ l_{i}^2=-1\ \text{for}\  i\in [1,N],\ \langle l_{i}, l_{j}\rangle =0\ \text{for all}\ i\neq j\] by Proposition \ref{prop: picard group of wdP}. Then by \cite[Thm 3.10]{DPT80} the effective cone of $\Pic \bar{X}' \otimes_{\mathbb{Z}}\mathbb{R}$ is generated by negative curves, hence each divisor $l_{i}$ can be written as a sum $\sum a_{j}R_{j}$, where $R_{j}$ is a $(-1)$ or $(-2)$-curve. We can define a morphism
\[
\psi:\Pic \bar{X}' \otimes_{\mathbb{Z}} \mathbb{R}\rightarrow M/F \otimes_{\mathbb{Z}} \mathbb{R}, l_{i} \mapsto \sum a_{j} R_{j}, 
\] where the image of $\sum a_{j}R_{j}$ is a choice of representation of $l_{i}$ as the sum of negative curves.
\begin{proposition}\label{prop: pic is iso to M/K} The map $\psi:\Pic \bar{X}'\otimes_{\mathbb{Z}}\mathbb{R}\rightarrow M/F\otimes_{\mathbb{Z}}\mathbb{R}$, is an isomorphism of $\Gal(\mathbb{F}_{q^{n}}/\mathbb{F}_{q})$-modules.
\end{proposition}
\begin{proof}
The map $\psi$ is clearly a $\Gal(\mathbb{F}_{q^{n}}/\mathbb{F}_{q})$-module homomorphism and surjective. Now it is sufficient to show $\psi$ is injective. Suppose $D\in \Pic \bar{X}'$ is in the kernel of $\psi$. Then $\psi(D)\in F$, hence $\langle D,D'\rangle=0$, for all negative curves on $\bar{X}$. As numerical equivalence is the same as linear equivalence in $\Pic \bar{X}'\otimes_{\mathbb{Z}}\mathbb{R}$, we have that $D=0$ in $\Pic \bar{X}'\otimes_{\mathbb{Z}}\mathbb{R}$.
\end{proof}
All possible actions of $\Gal(\bar{\mathbb{F}}_{q}/\mathbb{F}_{q})$ on $\Pic \bar{X}'$ are subsets of the possible actions of $\Aut(\Gamma)$ on $M/F$. One can also find the action of $\Gal(\bar{\mathbb{F}}_{q}/\mathbb{F}_{q})$ on the subset $\mathcal{R}\subset \Pic \bar{X}'$ of $(-2)$-curves as it is a linearly independent subset, hence one can determine the number of rational points on $X$ using Theorem \ref{theorem: Nathan Kaplan result}. Note that in Proposition \ref{prop: pic is iso to M/K} we base change to $\mathbb{R}$ so we are free to change basis. As the choice of basis has no effect on the trace of Frobenius and Picard group is torsion free this has no effect on our calculations in the future.
\end{algorithm}
\section{Toric varieties}\label{sec: toric}
In this section we introduce toric varieties and show that any toric singular del Pezzo surface over a finite field has to have a smooth rational point.
\begin{definition}
An \emph{algebraic torus} $T$ over a field $k$, is an algebraic group over $k$, such that $T$ becomes isomorphic to $\mathbb{G}_{m}^{n}$ over $\bar{k}$ for some $n\in \mathbb{N}$. 
\end{definition}
\begin{definition}A \emph{toric variety} is a normal variety $X$ over a field $k$, with a faithful action of an algebraic torus $T$ over $k$ which has a dense open orbit.
\end{definition}

\begin{proposition}\label{prop: toric over k bar then toric over k}
Let $X$ be a singular del Pezzo surface over a field $k$. Assume that $\bar{X}$ is toric over $\bar{k}$, then there exists a torus $T$ over $k$ with an action $T\times X\rightarrow X$ with a dense open orbit i.e.\ $X$ is toric over $k$. Moreover, if $k$ is a finite field then $X$ has a smooth rational point.
\end{proposition}
\begin{proof}
As $\bar{X}$ is toric we have a subtorus $T_{1}\subseteq \Aut(\bar{X})$. Then $T_{1}$ lies in a maximal torus $T'\subseteq \Aut(\bar{X})$, by Grothendieck's theorem \cite[SGA3, Exp. XIV, Thm 1.1]{SGA3} there exists a maximal torus $T\subseteq \Aut(X)$ such that $\bar{T}$ is maximal in $\Aut(\bar{X})$. As all maximal tori are conjugates and $T_{1}$ acts via dense open orbit, so does $T$. Hence, $X$ is toric over $k$. As $T$ acts freely and transitively on $U$, then $U$ is a $T$-torsor over $\Spec k$. If $k$ is a finite field by Lang's Theorem \cite[Thm 16.3]{B91} we have $\HH^{1}_{\et}(k,T)=0$ i.e.\ every $T$-torsor over $\Spec k$ is trivial. Then $U\cong T$ and $U$ is smooth and has a rational point.
\end{proof}
 \begin{proposition}\label{prop: toric del Pezzo surfaces} Let $X$ be a singular del Pezzo surface of degree $d$ over a finite field $k$. Suppose $X$ has singularity type $S$ over $\bar{k}$, and one of the following \begin{enumerate}
\item $d=7$ or $8$ with $S=A_{1}$,
\item $d=6$ with $S=A_{1}$ and $\bar{X}$ has 4 lines, $S=2A_{1}$ or $A_{2}+A_{1}$,
\item $d=5$ with $S=2A_{1}$ or $A_{2}+A_{1}$,
\item $d=4$ with $S=4A_{1},A_{2}+2A_{1}$ or $A_{3}+2A_{1}$,
\item $d=3$ with $S=3A_{2}$,
\end{enumerate} then $X$ has a smooth rational point
 \end{proposition}
 \begin{proof}
Derenthal gives a description of all possible singular del Pezzo surfaces which are toric \cite[Chap 1. \S 1.8]{U06}. These are exactly the cases in the Proposition. Using Proposition \ref{prop: toric over k bar then toric over k}, we can deduce $X$ has a smooth rational point.
 \end{proof}

\begin{notation}
Throughout the rest of the paper we fix the following notation.
 \begin{enumerate}
\item When drawing a graph of negative curves of a weak del Pezzo surface the filled in nodes correspond to $(-1)$-curves and the unfilled correspond to $(-2)$-curves.
\item  We fix $k$ be a finite field of size $q$.
\item For a singular del Pezzo surface $X$ we denote by $\pi:X'\rightarrow X$ the minimal desingularisation of $X$.
\item We denote by $\tr(\phi^{*})$ the trace of Frobenius on $\Pic \bar{X}'$ and $\tr(\phi^{*}\mid_{\mathcal{R}})$ the trace of Frobenius on the root sublattice $\mathcal{R}\subseteq \Pic \bar{X}'$ generated by $(-2)$-curves.
\end{enumerate}
\end{notation}
In the proceeding statements we do not explicitly state the singularity type. Rather, we have given detailed descriptions of the singularity type in Table \ref{table: dP6,7, and 8} for degrees 8, 7, 6 and Tables \ref{table: dp5}, \ref{table: dp4} for degrees 5 and 4 respectively. Moreover, we have the following key for our classification tables.
\begin{key}
\begin{itemize}
\item The first column of the tables label each class of singular del Pezzo surface.
\item The second column gives the Dynkin diagram type of singular points on $\bar{X}$.
\item The third column gives the  Dynkin diagram type of singular points on $\bar{X}$ which are invariant under the action of $\Gal(\bar{k}/k)$.
\item The fourth column gives the number of $(-1)$-curves on $\bar{X}$. This is of significance as two del Pezzo surfaces $X,X'$ of degree $d$ can have the same singularity type over $\bar{k}$ but a different number of $(-1)$-curves. However, the class of a singular del Pezzo surface over an algebraically closed field is uniquely determined by the Dynkin diagram types for the singular points and the number of $(-1)$-curves. To distinguish del Pezzo surfaces of degree $d$ with the same singularity type but with a different number of $(-1)$-curves, we use the notation $[-]'$ and $[-]''$.
\item The fifth column shows for which $q$ the given singularity type for $X$ posses a smooth $\mathbb{F}_{q}$ point. A $\checkmark$ placed in this entry if for any choice of $q$, this singularity type will have a smooth rational point. Moreover, we put a x in this column if this singularity type does not exist over a perfect field 
\item The sixth and final column gives the location for the proof of the particular singularity type having a smooth point over the stated fields, or a proof of why such a surface cannot exist. 
\end{itemize}
\end{key}

\section{Del Pezzo surfaces of degree 6, 7, and 8}\label{sec: dP7 and dP8}
In Table \ref{table: dP6,7, and 8} we give a classification of del Pezzo surfaces of degree 6, 7 and 8. Note that all singular Pezzo surfaces of degree 7 and 8 are toric, hence Proposition \ref{prop: toric over k bar then toric over k} shows that these surfaces always have a smooth rational point over a finite field.
\FloatBarrier
\begin{table}[H]
\caption{Classification of singular del Pezzo surfaces of degree 6,7 and 8}\label{table: dP6,7, and 8}
\begin{tabular}{|l|l|l|l|l|l|} \hline
Class & Singular points over  & Singular  & Lines & Smooth  & Proof\\ 
 &  algebraic closure & rational points& &Point &\\ \hline \hline
$8.1$ & $A_{1}$ & $A_{1}$ & $0$ & $\checkmark$ & Prop \ref{prop: toric del Pezzo surfaces} \\ \hline
$7.1$ & $A_{1}$ & $A_{1}$ & $2$ & $\checkmark$ & Prop \ref{prop: toric del Pezzo surfaces}\\ \hline
$6.1$ & $A_{1}$ & $A_{1}$ & $4$ & $\checkmark$  & Prop \ref{prop: toric del Pezzo surfaces} \\ \hline
$6.2$ & $A_{1}$ & $A_{1}$ & $3$ & $\checkmark$ & Prop \ref{prop: 6.2}\\ \hline
$6.3$ & $A_{2}$ & $A_{2}$ & $2$ & $\checkmark$ & Prop \ref{prop: 6.3} \\ \hline
$6.4$ & $2A_{1}$ & $\emptyset$ & $2$ & $\checkmark$ & Remark \ref{rem: singular del Pezzo with no singular point over ground field has a smooth point}\\ \hline
$6.5$ & $2A_{1}$ & $2A_{1}$ & $2$ & $\checkmark$ & Corollary \ref{cor: number of singular points not congruent to one mod q then have smooth point}\\ \hline
$6.6$ & $A_{1}+A_{2}$ & $A_{1}+A_{2}$ & $1$ & $\checkmark$ & Corollary \ref{cor: number of singular points not congruent to one mod q then have smooth point}\\ \hline
\end{tabular}
\end{table}

\begin{proposition}\label{prop: 6.2}
Let $X$ be a singular del Pezzo surface of type $6.2$ over $k$, then $X$ has a smooth rational point.
\end{proposition}
\begin{proof}
Consider the graph of negative curves $\Gamma$ \cite[Prop 8.3, Diagram 2]{CT88} on $\bar{X}'$. Using algorithm \ref{algorithm: determine traces on singular del pezzo}, we see that $\tr(\phi^{*})\geq 1$. Moreover, as there is only one $(-2)$-curve on $\bar{X}'$ we have $\tr(\phi^{*}\mid_{\mathcal{R}})=1$. Hence, $\#X(k)\geq q^2+1\geq 2$ as $q\geq 2$.
\end{proof}

\begin{proposition}\label{prop: 6.3}
Let $X$ be a singular del Pezzo surface of type $6.3$ over $k$, then $X$ has a smooth rational point.
\end{proposition}
\begin{proof}
Using algorithm \ref{algorithm: determine traces on singular del pezzo}, we see that the graph of negative curves $\Gamma$ \cite[Prop 8.3, Diagram 4]{CT88} on $\bar{X}'$ has an automorphism group $\Aut(\Gamma)\cong C_{2}$. If the action of Galois on $\bar{X}'$ is trivial then $\tr(\phi^{*})=4$ and $\tr(\phi^{*}\mid_{\mathcal{R}})=2$. If Galois acts via the non-trivial automorphism of $\Gamma$, then $\tr(\phi^{*})=2$ and $\tr(\phi^{*}\mid_{\mathcal{R}})=2$. Hence, $\#X(k)=q^2+1\geq 2$ as $q\geq 2$.
\end{proof}
\begin{corollary}\label{cor: degree 6, 7 and 8}
Every singular del Pezzo surface of degree 6, 7 or 8 over a finite field has a smooth rational point.
\end{corollary}
\section{Del Pezzo surfaces of degree 5}
We now deal with singular del Pezzo surfaces of degree 5.
\FloatBarrier
\begin{table}[H]
\caption{Classification of singular del Pezzo surfaces of degree 5}\label{table: dp5}
\begin{tabular}{|l|l|l|l|l|l|} \hline
Class & Singular points over  & Singular  & Lines & Smooth  & Proof\\ 
 &  algebraic closure & rational points& &Point &\\ \hline \hline
$5.1$ & $A_{1}$ & $A_{1}$ & $7$ & $\checkmark$ & Prop \ref{prop: 5.1}\\ \hline
$5.2$ & $2A_{1}$ & $\emptyset$ & $5$ & $\checkmark$ & Remark \ref{rem: singular del Pezzo with no singular point over ground field has a smooth point} \\ \hline
$5.3$ & $2A_{1}$ & $2A_{1}$ & $5$ & $\checkmark$ & Corollary \ref{cor: number of singular points not congruent to one mod q then have smooth point}/Prop \ref{prop: toric del Pezzo surfaces} \\ \hline
$5.4$ & $A_{2}$ & $A_{2}$ & $4$ & $\checkmark$ & Prop \ref{prop: 5.4}\\ \hline
$5.5$ & $A_{1}+A_{2}$ & $A_{1}+A_{2}$ & $3$ & $\checkmark$ & Corollary \ref{cor: number of singular points not congruent to one mod q then have smooth point}/Prop \ref{prop: toric del Pezzo surfaces}\\ \hline
$5.6$ & $A_{3}$ & $A_{3}$ & $2$ & $\checkmark$ & Prop \ref{prop: 5.6,5.7} \\ \hline
$5.7$ & $A_{4}$ & $A_{4}$ & $1$ & $\checkmark$ & Prop \ref{prop: 5.6,5.7}\\ \hline
\end{tabular}
\end{table}
\FloatBarrier

\begin{proposition}\label{prop: 5.1}
Let $X$ be a singular del Pezzo surface of type $5.1$ over $k$, then $X$ has a smooth rational point.
\end{proposition}
\begin{proof}
Consider the graph of negative curves $\Gamma$ on $\bar{X}'$ as shown in \cite[Prop 8.5, Diagram 1]{CT88}
\begin{center}
\begin{tikzpicture}
\draw (-4,0) -- (2,0);
\draw (2,0) -- (4,0);
\draw (-4,0) -- (-2,2);
\draw (-2,2) -- (2,2);
\draw (2,2) -- (4,0);
\draw (-4,0) -- (-2,-2);
\draw (-2,-2) -- (2,-2);
\draw (2,-2) -- (4,0);
\node[draw, circle,fill=white,label = left:$l_{1}-l_{2}$]at (-4,0){};
\node[draw, circle,fill=black,label = $l_{13}$]at (-2,2){};
\node[draw, circle,fill=black,label = $l_{2}$]at (-2,0){};
\node[draw, circle,fill=black,label = $l_{12}$]at (2,0){};
\node[draw, circle,fill=black,label =right: $l_{34}$]at (4,0){};
\node[draw, circle,fill=black,label = $l_{3}$]at (2,2){};
\node[draw, circle,fill=black,label = $l_{12}$]at (-2,-2){};
\node[draw, circle,fill=black,label = $l_{4}$]at (2,-2){};
.
\end{tikzpicture}
\end{center}
 Any possible automorphism of $\Gamma$ will fix the node $l_{34}$ i.e.\ the action of $\Gal(\bar{k}/k)$ on $\bar{X}'$ will always fix a $(-1)$-curve. As this curve does not intersect the $(-2)$-curve $l_{1}-l_{2}$ we can deduce using Proposition \ref{prop: point lying away from -2 curve} that $X$ has a smooth rational point.
\end{proof}
\begin{proposition}\label{prop: 5.4}
Let $X$ be a singular del Pezzo surface of type $5.4$ over $k$, then $X$ has a smooth rational point.
\end{proposition}
\begin{proof}
Consider the graph of negative curves $\Gamma$ on $\bar{X}'$ as shown in \cite[Prop 8.5, Diagram 3]{CT88}
\begin{center}
\begin{tikzpicture}
\draw (-4,0) -- (-2,0);
\draw (-2,0) -- (0,0);
\draw (0,0) -- (2,0);
\draw (2,0) -- (3,-2.5);
\draw (2,0) -- (3,2.5);
\node[draw, circle,fill=black,label = $l_{4}$]at (-4,0){};
\node[draw, circle,fill=black,label = $l_{14}$]at (-2,0){};
\node[draw, circle,fill=white,label = $l_{1}-l_{2}$]at (0,0){};
\node[draw, circle,fill=white,label = $l_{2}-l_{3}$]at (2,0){};
\node[draw, circle,fill=black,label = $l_{3}$]at (3,2.5){};
\node[draw, circle,fill=black,label = $l_{12}$]at (3,-2.5){};
\end{tikzpicture}
\end{center}
We see the $(-1)$-curve $l_{1}-l_{2}$ is fixed under the action of Galois and only intersects one $(-2)$-curve, hence we can apply Proposition \ref{prop: point lying away from -2 curve} to deduce $X$ has a smooth rational point.
\end{proof}
\begin{proposition}\label{prop: 5.6,5.7}
Let $X$ be a singular del Pezzo surface of type $5.6$ or $5.7$ over $k$, then $X$ has a smooth rational point.
\end{proposition}
\begin{proof} Denote by $\Gamma_{1}$ (resp. $\Gamma_{2}$) the graph of negative curves on $\bar{X}'$ if $X$ is of type 5.7 (resp. 5.8). Then
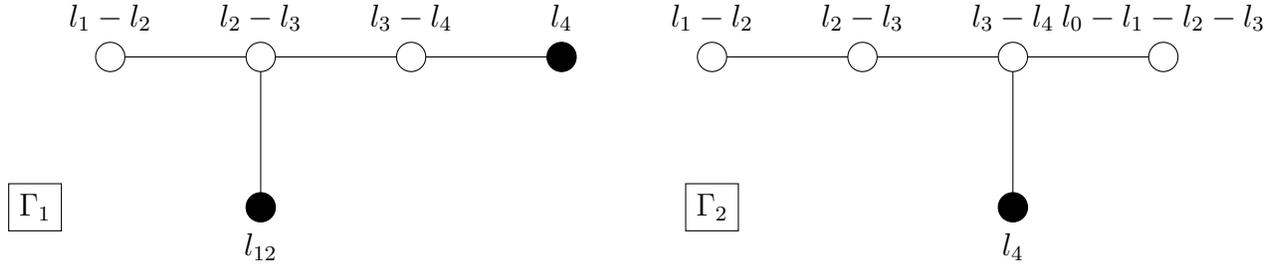
\begin{center}
\begin{tikzpicture}
\draw (0,0) -- (2,0);
\draw (2,0) -- (4,0);
\draw (4,0) -- (6,0);
\draw (2,0) -- (2,-2);
\draw (8,0) -- (10,0);
\draw (10,0) -- (12,0);
\draw (12,0) -- (14,00);
\draw (12,0) -- (12,-2);
\node[draw, circle,fill=white,label=$l_{1}-l_{2}$]at (0,0){};
\node[draw]  [label=right:] at (-1,-2) {$\Gamma_{1}$}; 
\node[draw]  [label=right:] at (8,-2) {$\Gamma_{2}$}; 
\node[draw, circle,fill=white,label=$l_{2}-l_{3}$]at (2,0){};
\node[draw, circle,fill=black,label=below:$l_{12}$]at (2,-2){};
\node[draw, circle,fill=white,label=$l_{3}-l_{4}$]at (4,0){};
\node[draw, circle,fill=black,label=$l_{4}$]at (6,0){};
\node[draw, circle,fill=white,label = $l_{1}-l_{2}$]at (8,0){};
\node[draw, circle,fill=white,label = $l_{2}-l_{3}$]at (10,0){};
\node[draw, circle,fill=white,label = $l_{3}-l_{4}$]at (12,0){};
\node[draw, circle,fill=black,label = below:$l_{4}$]at (12,-2){};
\node[draw, circle,fill=white,label = $l_{0}-l_{1}-l_{2}-l_{3}$]at (14,0){};
\end{tikzpicture}
\end{center} 
By observing the $(-1)$-curve $l_{1}-l_{2}$ will be fixed in both graphs and does not intersect any $(-2)$-curves we can deduce using Proposition \ref{prop: point lying away from -2 curve} that $X$ has a smooth rational point.
\end{proof}
\begin{corollary}\label{cor: degree 5}
Every singular del Pezzo surface of degree 5 over a finite field has a smooth rational point.
\end{corollary}
\section{Del Pezzo surfaces of degree 4}\label{Sec: dp4 section}
This section is broken into two parts. The first deals with odd characteristic where we can give very geometric reasoning on why singular del Pezzo surfaces of degree 4 have a smooth point. The second uses algorithm \ref{algorithm: determine traces on singular del pezzo}, but works in arbitrary characteristic.

\begin{longtable}{|l|l|l|l|l|l|}
\caption{Classification of singular del Pezzo surfaces of degree 4}\label{table: dp4}\\\hline
Class & Singular points over  & Singular  & Lines & Smooth  & Proof\\ 
 &  algebraic closure & rational points& &Point &\\ \hline \hline
$4.1$ & $A_{1}$ & $A_{1}$ & $12$ & $\checkmark$ & Prop \ref{prop: 4.1}  \\ \hline
$4.2$ & $[2A_{1}]'$ & $\emptyset$ & $9$ & $\checkmark$ & Remark \ref{rem: singular del Pezzo with no singular point over ground field has a smooth point}\\ \hline
$4.3$ & $[2A_{1}]'$ & $2A_{1}$ & $9$ & $\checkmark$ & Corollary \ref{cor: number of singular points not congruent to one mod q then have smooth point}\\ \hline
$4.4$ & $[2A_{1}]''$ & $\emptyset$ & $8$ & $\checkmark$ & Remark \ref{rem: singular del Pezzo with no singular point over ground field has a smooth point}\\ \hline
$4.5$ & $[2A_{1}]''$ & $2A_{1}$ & $8$ & $\checkmark$ & Corollary \ref{cor: number of singular points not congruent to one mod q then have smooth point}\\ \hline
$4.6$ & $A_{2}$ & $A_{2}$ & $8$ & $\checkmark$ & Prop \ref{prop: 4.6}\\ \hline
$4.7$ & $3A_{1}$ & $\emptyset$ & $6$ & $\checkmark$ &  Remark \ref{rem: singular del Pezzo with no singular point over ground field has a smooth point}\\ \hline
$4.8$ & $3A_{1}$ & $A_{1}$ & $6$ & $\checkmark$ & Prop \ref{proposition: 4.8}\\ \hline
$4.9$ & $3A_{1}$ & $3A_{1}$ & $6$ & $\checkmark$ & Prop \ref{prop: 4.9,4.11,4.20,4.24} \\ \hline
$4.10$ & $A_{1}+A_{2}$ & $A_{1}+A_{2}$ & $6$ & $\checkmark$ & Corollary \ref{cor: number of singular points not congruent to one mod q then have smooth point}\\ \hline
$4.11$ & $[A_{3}]'$ & $A_{3}$ & $5$ & $\checkmark$ & Prop \ref{prop: 4.9,4.11,4.20,4.24}  \\ \hline
$4.12$ & $[A_{3}]''$ & $A_{3}$ & $4$ & \checkmark & Prop \ref{prop:4.12} \\ \hline
$4.13$ & $A_{1}+A_{3}$ & $A_{1}+A_{3}$ & $3$ & $\checkmark$ & Corollary \ref{cor: number of singular points not congruent to one mod q then have smooth point} \\ \hline
$4.14$ & $A_{2}+2A_{1}$ & $A_{2}$ & $4$ & $\checkmark$ & Prop \ref{prop:4.14,4.15}\\ \hline
$4.15$ & $A_{2}+2A_{1}$ & $A_{2}+2A_{1}$ & $4$ & $\checkmark$ &Prop \ref{prop:4.14,4.15}\\ \hline
$4.16$ & $4A_{1}$ & $\emptyset$ & $4$ & $\checkmark$ & Remark \ref{rem: singular del Pezzo with no singular point over ground field has a smooth point}\\ \hline
$4.17$ & $4A_{1}$ & $A_{1}$ & $4$ & x  & Prop \ref{proposition: 4.18} \\ \hline
$4.18$ & $4A_{1}$ & $2A_{1}$ & $4$ & $\checkmark$ & Corollary \ref{cor: number of singular points not congruent to one mod q then have smooth point}/Prop \ref{prop: toric del Pezzo surfaces} \\ \hline
$4.19$ & $4A_{1}$ & $4A_{1}$ & $4$ & $\checkmark$ & Prop \ref{prop:4.19}/Prop \ref{prop: toric del Pezzo surfaces} \\ \hline
$4.20$ & $A_{4}$ & $A_{4}$ & $3$ & $\checkmark$ & Prop \ref{prop: 4.9,4.11,4.20,4.24} \\ \hline
$4.21$ & $D_{4}$ & $D_{4}$ & $2$ & $\checkmark$ & Prop \ref{proposition: 4.21}\\ \hline
$4.22$ & $2A_{1}+A_{3}$ & $A_{3}$ & $2$ & $\checkmark$ & Prop \ref{prop:4.22,4.23} \\ \hline
$4.23$ & $2A_{1}+A_{3}$ & $2A_{1}+A_{3}$ & $2$ & $\checkmark$ & Prop \ref{prop:4.22,4.23} \\ \hline
$4.24$ & $D_{5}$ & $D_{5}$ & $1$ & $\checkmark$ & Prop \ref{prop: 4.9,4.11,4.20,4.24}\\ \hline
\end{longtable}

\subsection{Odd characteristic}\label{subsec: dP4 odd char}
Throughout Subsection \ref{subsec: dP4 odd char} we shall assume $q$ is coprime to 2 i.e.\ $k$ is of odd characteristic.
\begin{lemma}[{\cite[Table 1]{F08}}]\label{lem: Point count on quadric surfaces in P3}
Let $Q\subset \mathbb{P}^{3}_{k}$ a quadric surface. Then the number of rational points on $Q$ is as follows:
\[
\begin{tabular}{|l|l|l|}\hline
Rank & Type & $Q(\mathbb{F}_{q})$ \\\hline\hline
1 & repeated plane & $q^2+q+1$  \\\hline
2 & pair of distinct planes & $2q^2+q+1$\\\hline
2 & line & $q+1$\\\hline
3 & quadric cone & $q^2+q+1$\\\hline
4 & hyperbolic quadric & $(q+1)^2$\\\hline
4 & elliptic quadric & $q^2+1$\\\hline
\end{tabular}
\]
\end{lemma}
\begin{proposition}\label{prop: dP4 odd char}
Let $X$ be a singular del Pezzo of degree 4 over $k$ with at least one singular rational point. Then $\#X(\mathbb{F}_{q})\geq q^2-2q+1$.
\end{proposition}
\begin{proof}
By \cite[Prop 2.1]{CTSSD87}, we can write $X$ as \[
X = \begin{cases} 
x_{0}x_{1}-g(x_{1},x_{2},x_{3},x_{4})&=0\\
				f(x_{1},x_{2},x_{3},x_{4})&=0
	\end{cases} \ \ \subseteq \mathbb{P}^{4}
\] where $f$ and $g$ are quadratic forms, $f$ is of rank at least 3 and $X$ has a singular point at $(x_{0},x_{1},x_{2},x_{3},x_{4})=(1,0,0,0,0)$. Considering the affine open $U:=X\backslash \mathbb{V}(x_{1})$ we can define an isomorphism from $U$ to the affine scheme $V \subset \mathbb{A}^{3}$, defined by $f(1,x_{2},x_{3},x_{4})=0$ via \begin{align*}
U &\rightarrow V, &&(x_{0},x_{2},x_{3},x_{4})\mapsto (x_{2},x_{3},x_{4}),\\
V &\rightarrow U, &&(x_{2},x_{3},x_{4}) \mapsto (g(1,x_{2},x_{3},x_{4}),x_{2},x_{3},x_{4}).
\end{align*} 
Let $\mathbb{V}:=\mathbb{V}(f(x_{1},x_{2},x_{3},x_{4}))\subset \mathbb{P}^{3}$ be the compactification of $V$ inside $\mathbb{P}^{3}$ and denote by $H$ the hyperplane $H:=\mathbb{V}(x_{1})$. The hyperplane section $C:=\mathbb{V}\cap H$ is a (possibly singular) conic. One can deduce $\# X(k)=\#\mathbb{V}(k)-\#C(k)+\#\left(X\cap \mathbb{V}(x_{1})\right)(k)$. If all the singular points of $X$ lie on the plane $\mathbb{V}(x_{1})$ then $\#X(k)\geq q^2-2q+1$, or $\#X(\mathbb{F}_{q})\geq q^2-q+1$, otherwise.
\end{proof}
\begin{corollary}\label{cor: dP4 odd char not 4}
Let $X$ be a singular del Pezzo of degree 4 over $k$ where $X$ has $\delta>0$ singular rational points. If $\delta \neq 4$ then $X$ has a smooth point.
\end{corollary}
\begin{proof}
Note that by the classification of singular del Pezzo surfaces of degree $4$ we have $\delta\leq 4$. By Proposition \ref{prop: dP4 odd char} and the fact $q\geq 3$ we have that number of smooth points on $X$ is at least $q^2-2q+1-\delta \geq 4-\delta$. Hence, we have a smooth point under the assumption $\delta \neq 4$.
\end{proof}
\begin{corollary}\label{cor: dP4 odd char is 4}
 Let $X$ be a singular del Pezzo of degree 4 over $k$ and suppose $X$ has $4$ singular rational points. Then $X$ has a smooth rational point.
\end{corollary}
\begin{proof}
Keeping notation as in Proposition \ref{prop: dP4 odd char}, if all singular points lie on the hyperplane $\mathbb{V}(x_{1})$ then $\#X(k)\geq q^2-2q+4$, hence for $q\geq 3$ we always have a smooth point. If there exists a singular point lying away from $\mathbb{V}(x_{1})$ then $f$ must have rank 3, hence $\#X(k)\geq q^2-q+1$. As $q^2-q-3>0$ for $q\geq 3$ we have that $X$ must have a smooth rational point.
\end{proof}
\subsection{General characteristic}
\begin{lemma}\label{lemma: 4.1} Let $X$ be a singular del Pezzo of of type $4.1$ over $\mathbb{F}_{2}$. Denote by $X'$ the minimal desingularisation of $X$ then the trace of Frobenius, $\tr(\phi^{*})$ on $X'$ is not $-1$ or $-2$.
\end{lemma}
\begin{proof}
Suppose $\tr(\phi^{*})=-2$, then as $X$ has an $A_{1}$ singularity we must have $\tr(\phi\mid_{\mathcal{R}})=1$, hence $\#X(\mathbb{F}_{2})=-1$ which is a contradiction. If $\tr(\phi^{*})=-1$, then the order of the action of Frobienus on $\bar{X}'$ must be $6$. Considering the graph of negative curves $\Gamma$ \cite[Prop 6.1, Diagram 1]{CT88} on $\bar{X'}$, we see that all elements of order 6 of $\Aut(\Gamma)$ fix a $(-1)$-curve on $\bar{X}'$, hence $\#X'(\mathbb{F}_{2})\geq 5$. However, if $\tr(\phi^{*})=-1$ then $\#X'(\mathbb{F}_{2})=3$ so we have a contradiction.
\end{proof}
\begin{proposition}\label{prop: 4.1} Let $X$ be a singular del Pezzo of type 4.1 over $k$, then $X$ has a smooth rational point.
\end{proposition}
\begin{proof}
If we consider all possible traces for the action of Frobenius on $\bar{X}'$ we see that $\tr(\phi^{*})\geq -2$. This is done by simply enumerating all conjugacy classes of $W(D_{4})$. If $\tr(\phi^{*})=-2$ and $q\geq 4$ we have $\#X(k)\geq 5$, hence $X$ has a smooth rational point. The case where $q=3$ is dealt with in Corollary \ref{cor: dP4 odd char not 4}. Using Lemma \ref{lemma: 4.1} we see that if $q=2$ then $\tr(\phi^{*})\geq 0$, hence $\#X(k)\geq 3$ and we can deduce $X$ has a smooth rational point in this case also.
\end{proof}
\begin{proposition}\label{prop: 4.6}
Let $X$ be a singular del Pezzo of type $4.6$ over $k$. Then $X$ has a smooth rational point.
\end{proposition}
\begin{proof}
Consider the graph of negative curves $\Gamma$, on $\bar{X}'$ \cite[Prop 6.1, Diagram 4]{CT88}. By running algorithm \ref{algorithm: determine traces on singular del pezzo} we see that $\tr(\phi^{*})=0,2$ or 4, respectively $\tr(\phi^{*}\mid_{\mathcal{R}})=0,2$ or 2. We can conclude that $\#X(\mathbb{F}_{q})\geq q^2+1\geq 5$ for $q\geq 2$, hence $X$ has a smooth rational point.
\end{proof}
\begin{proposition}\label{proposition: 4.8} Let $X$ be a singular del Pezzo surface of type $4.8$ over $k$. Then $X$ has a smooth rational point.
\end{proposition}
\begin{proof}
In this case one can run algorithm \ref{algorithm: determine traces on singular del pezzo} and see that $\tr(\phi^{*})=0,2,$ or 4 respectively $\tr(\phi^{*}\mid_{\mathcal{R}})=1,1$ or 3. Hence, $\#X(\mathbb{F}_{q})\geq q^2-q+1>1$ and $X$ has a smooth rational point.
\end{proof}
\begin{proposition}\label{prop: 4.9,4.11,4.20,4.24} Let $X$ be a singular del Pezzo surface of type $4.9,4.11,4.20$ or $4.24$ over $k$. Then $X$ has a smooth rational point.
\end{proposition}
\begin{proof}
Consider the graph of negative curves $\Gamma$, on $\bar{X}'$, as shown in \cite[Prop 6.1, Diagrams 7,12,15]{CT88}. Then one can see that there is a always a $(-1)$-curve fixed by Galois which intersect at most two $(-2)$-curves. Using Proposition \ref{prop: point lying away from -2 curve} we can conclude that $X$ has a smooth rational point.
\end{proof}
\begin{proposition}\label{prop:4.12} Let $X$ be a singular del Pezzo surface of type $4.13$ over $k$. Then $X$ has a smooth rational point.
\end{proposition}
\begin{proof}
In this case one can run algorithm \ref{algorithm: determine traces on singular del pezzo} and see that $\tr(\phi^{*})=2,2$ or 4 respectively $\tr(\phi^{*}\mid_{\mathcal{R}})=1,2$ or 1. Hence, $\#X(k)\geq q^2-q+1>1$ and $X$ has a smooth rational point.
\end{proof}
\begin{proposition}\label{prop:4.14,4.15} Let $X$ be a singular del Pezzo surface of type $4.15$ or $4.16$ over $k$. Then $X$ has a smooth rational point.
\end{proposition}
\begin{proof}
Consider the graph of negative curves $\Gamma$, on $\bar{X}'$, as shown in \cite[Diagram 10, Prop 6.1]{CT88}. In the case where $X$ is of type $4.16$ it is easy to see that the action of $\Gal(\bar{k}/k)$ on $\Gamma$ is trivial. Moreover, in this case there exists a $(-1)$-curve which is defined over $k$ and intersects one $(-2)$-curve, hence we can apply Proposition \ref{prop: point lying away from -2 curve} to deduce there is a smooth rational point on $X$. If we are in the case of $4.15$ then $\Aut(\Gamma)\cong C_{2}$, and as we are in case $4.15$ the Galois action on $\bar{X}'$ must be non-trivial. Then we must have $\tr(\phi^{*})=0$ and $\tr(\phi\mid_{\mathcal{R}})=0$ and $\#X(k)=q^2+1>3$ as $q\geq 2$.
\end{proof}
\begin{proposition}\label{proposition: 4.18} 
There are no del Pezzo surfaces of degree 4 of type 4.18 over a perfect field $K$.
\end{proposition}
\begin{proof}
Considering the graph of negative curves \cite[Prop 6.1, Diagram 9]{CT88}, on the minimal desingularisation $\bar{X}'$ as shown below.
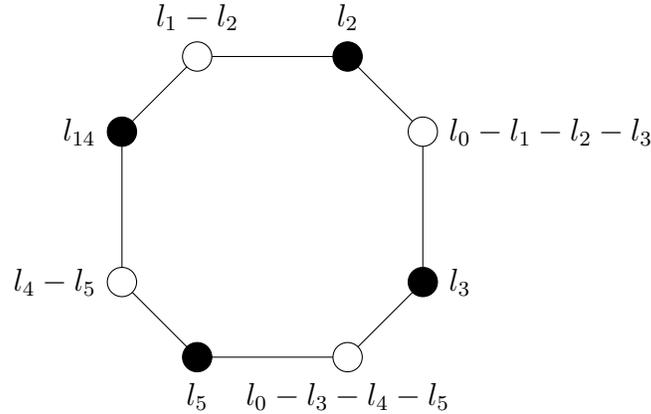
\begin{center}
\begin{tikzpicture}
\draw (-2,0) -- (-1,1);
\draw (-2,0) -- (-2,-2);
\draw (-2,-2) -- (-1,-3);
\draw (-1,-3) -- (1,-3);
\draw (1,-3) -- (2,-2);
\draw (-1,1) -- (1,1);
\draw (1,1) -- (2,0);
\draw (2,0) -- (2,-2);

\node[draw, circle,fill=black,label=left:$l_{14}$]at (-2,0){};
\node[draw, circle,fill=white,label=$l_{1}-l_{2}$]at (-1,1){};
\node[draw, circle,fill=black,label =$l_{2}$]at (1,1){};
\node[draw, circle,fill=white,label=right:$l_{0}-l_{1}-l_{2}-l_{3}$]at (2,0){};
\node[draw, circle,fill=black,label=right:$l_{3}$]at (2,-2){};
\node[draw, circle,fill=white,label =below: $l_{0}-l_{3}-l_{4}-l_{5}$]at (1,-3){};
\node[draw, circle,fill=black,label =below: $l_{5}$]at (-1,-3){};
\node[draw, circle,fill=white,label = left:$l_{4}-l_{5}$]at (-2,-2){};
\end{tikzpicture}
\end{center} 
We see that for example if the $(-2)$-curve $l_{1}-l_{2}$ is defined over $K$ then the other negative curves must also be defined over $K$, as there are no automorphisms of the above graph i.e.\ the Galois action on $\bar{X}'$ must be trivial. Hence, the Galois action on $\bar{X}$ must also be trivial but this singularity type requires a non trivial Galois action on $\bar{X}$ as three singular points are permuted.
\end{proof}

\begin{proposition}\label{prop:4.19} Let $X$ be a singular del Pezzo surface of type $4.20$ over $k$. Then $X$ has a smooth rational point.
\end{proposition}
\begin{proof}
Using Corollary \ref{cor: number of singular points not congruent to one mod q then have smooth point} we see that for $q\neq 3$, there always exits a smooth rational point on $X$. For $q=3$, Corollary \ref{cor: dP4 odd char is 4} allows us to deduce that there is smooth rational point on $X$ in this case also.
\end{proof}
\begin{proposition}\label{proposition: 4.21} Let $X$ be a singular del Pezzo surface of type $4.21$ over $k$. Then $X$ has a smooth rational point.
\end{proposition}
\begin{proof}
One can consider the graph of negative curves, $\Gamma$ on $\bar{X}'$, as shown in \cite[Prop 6.1, Diagram 13]{CT88}. Here we can see that $\Aut(\Gamma)\cong C_{2}$. Denote by $\sigma\in \Aut(\Gamma)$, the non-trivial element. If $\Gal(\bar{k}/k)$ acts trivially on $\Gamma$, then $\#X'(k)=q^2+6q+1$, and $\#X(k)\geq q^2+2q+1$, hence $X$ has a smooth point. If the action of $\Gal(\bar{k}/k)$ on $\Gamma$, factors through the subgroup generated by $\sigma$, then $\tr(\phi^{*})=2$ and $\tr(\phi\mid_{\mathcal{R}})=2$. Hence, $\#X(k)=q^2+1> 2$, as $q\geq 2$.
\end{proof}
\begin{proposition}\label{prop:4.22,4.23} Let $X$ be a singular del Pezzo surface of type $4.22$ or $4.23$ over $k$. Then $X$ has a smooth rational point.
\end{proposition}
\begin{proof}
Denote by Consider the graph $\Gamma$ of negative curves on the minimal desingularisation $X'$ of $X$ as shown in \cite[Diagram 14, Prop 6.1]{CT88}. In the case where $X$ is of type $4.24$ it is easy to see that the action of $\Gal(\bar{k}/k)$ on $\Gamma$ is trivial. Moreover, in this case there exists a $(-1)$-curve which is defined over $k$ and intersects one $(-2)$-curve, hence we can apply Proposition \ref{prop: point lying away from -2 curve} to deduce there is a smooth rational point on $X$. If we are in the case of $4.23$ then $\Aut(\Gamma)\cong C_{2}$, and as we are in case $4.23$ the Galois action on $\bar{X}'$ must be non-trivial. Then we must have $\tr(\phi^{*})=2$ and $\tr(\phi\mid_{\mathcal{R}})=1$, hence $\#X(k)=q^2+q+1>1$ as $q\geq 2$.
\end{proof}
\begin{corollary}\label{cor: degree 4}
Every singular del Pezzo surface of degree 4 over a finite field has a smooth rational point.
\end{corollary}
\section{Del Pezzo surfaces of degree 3}
\begin{theorem}[{\cite[Thm 2]{K02}}]\label{thm: Kollar} Let $K$ be a perfect field and $X\subset \mathbb{P}^{n+1}$ an irreducible cubic hypersurface of dimension $\geq 2$ over $K$ which is not a cone over a $(n-1)$-dimensional cubic. Then $X$ has a $K$-point if and only if $X$ has a smooth $K$-point.
\end{theorem}
\begin{corollary}\label{cor: dp3}
If $X$ is a singular del Pezzo of degree 3 over $k$, then $X$ always has a smooth rational point.
\end{corollary}
\begin{proof}
By Corollary \ref{theorem: Nathan Kaplan result} there always exists a $k$-point on $X$, then immediately from Theorem \ref{thm: Kollar} $X$ has a smooth rational point.
\end{proof}
\begin{remark}\label{rem: C-W cubic hypersurface}
To prove Corollary \ref{cor: dp3} where $X$ is a general cubic hypersurface, one can use Chevalley-Warning.
\end{remark}
\section{Del Pezzo surfaces of degree 2}\label{sec: del Pezzo of degree 2}
Any singular del Pezzo surface $X$ of degree 2 is of the form \[
X: w^{2}+wG_{2}(x,y,z)+G_{4}(x,yz) \subset \mathbb{P}(1,1,1,2)
\] where $G_{i}$ is a homogeneous polynomial of degree $i$.
\begin{definition} Let $S$ be a variety over a field $K$. We call the morphism generated by the linear system $\mid -K_{S}\mid $, the anticanonical morphism.
\end{definition} 
\begin{remark} Let $X$ be a singular del Pezzo of degree 2 over a field $K$ and denote by $f$ anticanonical map associated to $X$ \[
f: X\rightarrow \mathbb{P}^{2}_{K},[x:y:z:w]\mapsto [x:y:z].
\]This realises $X$ as a double cover of $\mathbb{P}^{2}_{K}$. 
If $\charr(K)\neq 2$ then by completing the square we can assume that $X$ is a of the form $w^{2}=G_{4}(x,y,z)$ where $G_{4}$ is a quartic curve. In this case the double cover has a \emph{branch locus} $B:G_{4}(x,y,z)=0$. If $\charr(K)=2$ the morphism $f$ can be inseparable. In the cases $f$ is separable the double cover has a \emph{branch locus} defined by the plane conic $B:G_{2}(x,y,z)=0$ (which can be reducible or non-reduced). 
In both cases we take the \emph{ramification curve} to be the reduced subscheme of $f^{-1}(B)$ denoted by $f^{-1}(B)_{\text{Red}}$ i.e.\ the ramification curve $R$ is given by \[
R:\begin{cases}
G_{4}(x,y,z) = 0 &\text{ if } \charr(K)\neq 2,\\
G_{2}(x,y,z)=0,w^2=G_{4}(x,y,z)&\text{ if } \charr(K)= 2.\\
\end{cases}
\] 
\end{remark}
\subsection{Separable anticanonical morphism}
In this section we give cases where the anticanonical morphism associated to a singular del Pezzo surface of degree 2 is separable. The proof is inspired by \cite[Prop 3.1]{DM23}. The statement in \cite{DM23} is only for smooth del Pezzo surfaces of degree 2, we show this proof also works for their singular counterparts.
\begin{lemma}\label{lemma: separable condition}
Let $K$ be a field and $X,Y$ integral separated schemes of finite type over $K$. Suppose $f:X\rightarrow Y$ is a finite morphism and $K^{\text{sep}}(X)/K^\text{sep}(Y)$ is separable, then $K(X)/K(Y)$ is separable.
\end{lemma}
\begin{proof}
As $K^{\text{sep}}(X)/K^{\text{sep}}(Y)$ and $K^{\text{sep}}(Y)/K(Y)$ are separable $K^{\text{sep}}(X)/K(Y)$ is separable. Then the inclusion of fields $K(Y)\subset K(X)\subset K^{\text{sep}}(X)$ implies via \cite[\S 4, Thm 4.5]{L02} that $K(X)/K(Y)$ is separable. 
\end{proof}

\begin{lemma}\label{lemma: Picard rank of singular del Pezzo surface} Let $X$ be a singular del Pezzo surface of degree $d$ over an algebraically closed field $K$ with minimal desingularisation $\pi:X'\rightarrow X$. If $n:=\#(-2)$-curves on $X'$, then the rank of $\Pic X$ is $10-d-n$.
\end{lemma}
\begin{proof}
We have an exact sequence \cite[Prop 1]{B13} \begin{equation}\label{eq: Picard groupof singular del Pezzo}
0\rightarrow \Pic X\xrightarrow[]{\pi^{*}} \Pic X'\rightarrow \mathcal{R}^{\vee}\rightarrow \Br X \rightarrow \Br X'.
\end{equation} As $X'$ is rational and $K$ is algebraically closed $\Br X'=0$. Moreover, $\Br X'$ is torsion  \cite[Prop 4]{B13} so tensoring (\ref{eq: Picard groupof singular del Pezzo}) with the flat $\mathbb{Z}$-module $\mathbb{Q}$ results in the exact sequence of $\mathbb{Q}$-vector spaces\[
0\rightarrow \Pic X\otimes \mathbb{Q}\xrightarrow[]{\pi^{*}} \Pic X'\otimes \mathbb{Q}\rightarrow\mathcal{R}^{\vee}\otimes \mathbb{Q}\rightarrow 0.
\] As tensoring preserves direct sums \[
\rank_{\mathbb{Z}}\Pic X+\rank_{\mathbb{Z}}\mathcal{R}=\rank_{\mathbb{Z}} \Pic X'.
\] We have $X'$ is the blow up of $9-d$ points of $\mathbb{P}^2$, showing $\Pic X'$ has rank $10-d$ and as the set of $(-2)$-curves in $ \Pic X'$ are linearly independent $\rank_{\mathbb{Z}}\mathcal{R}=n$. The statement now follows.
\end{proof}

\begin{proposition}\label{cor: separable morphism for dP2s}  Denote by $X$ a singular del Pezzo of degree 2 over a perfect field $K$. Let $X'$ be the minimal desingularisation of $X$. If the number of $(-2)$-curves on $\bar{X}'$ less than 7, then the anticanonical morphism $f: X\rightarrow \mathbb{P}^{2}_{k}$ is separable.
\end{proposition}
\begin{proof} We can assume that $K$ is algebraically closed by Lemma \ref{lemma: separable condition}. If the characteristic of $K$ is not $2$ then this is clear. Assume now that $\charr(K)= 2$ and $f$ is purely inseparable. Pick  a prime $l\gg 0$ and consider the Kummer sequence of étale sheaves associated to $X$ \[
0\rightarrow \mu_{l^{n},X}\rightarrow \mathbb{G}_{m,X}\rightarrow \mathbb{G}_{m,X}\rightarrow 0
\] where $n\in \mathbb{N}$. Applying étale cohomology gives an exact sequence \[\begin{tikzcd}
0 \arrow[r] & \HH^{1}(X,\mu_{l^{n}}) \arrow[r] & \HH^{1}(X,\mathbb{G}_{m}) \ar[draw=none]{d}[name=X, anchor=center]{}\arrow[r,"x\mapsto x^{\eln}"]
    & \HH^{1}(X,\mathbb{G}_{m}) \ar[rounded corners,
            to path={ -- ([xshift=2ex]\tikztostart.east)
                      |- (X.center) \tikztonodes
                      -| ([xshift=-2ex]\tikztotarget.west)
                      -- (\tikztotarget)}]{dll}[at end]{\ }       &               &    \\
            & \HH^{2}(X,\mu_{l^{n}}) \arrow[r,] & \HH^{2}(X,\mathbb{G}_{m}) \arrow[r,"x\mapsto x^{\eln}"] & \HH^{2}(X,\mathbb{G}_{m}) \arrow[r]  & \HH^{3}(X,\mu_{l^{n}}) \arrow[r] & \dotsc
\end{tikzcd} \] Using $\Pic X=\HH^{1}(X,\mathbb{G}_{m})$ and $\Br X=\HH^{2}(X,\mathbb{G}_{m})$ the exact sequence reduces to 
\[
0\rightarrow \Pic X/(\Pic(X)[l^{n}])\rightarrow \HH^{2}(X,\mu_{l^{n}})\rightarrow \Br X[l^{n}]\rightarrow 0
\] and as we chose $l\gg0$ we have $\Br X[l^{n}]=0$ for all $n\in \mathbb{N}$ \cite{B13}. Taking inverse limits we get $\Pic X\times \mathbb{Z}_{l}\cong \HH^{2}(X,\mathbb{Z}_{l})$. If $f$ is purely inseparable then $f$ is a homeomorphism in the étale topology, however the rank of $\HH^{2}(X,\mathbb{Z}_{l})$ is greater than 1 by Lemma \ref{lemma: Picard rank of singular del Pezzo surface} and $\HH^{1}(\mathbb{P}^2,\mathbb{Z}_{l})$ has rank 1.
\end{proof}
\subsection{Conic bundles}\label{subsec: conic bundles}
In this subsection we introduce conic bundles and their relation to singular del Pezzo surfaces of degree 2.
\begin{definition}[{Conic bundles}]
A \emph{conic bundle} is a smooth projective surface $S$ with a dominant morphism $S\rightarrow \mathbb{P}^1$, whose geometric generic fibre is a smooth integral curve of genus 0.
\end{definition}
\begin{remark}
We do not require all fibres to be plane conics in our definition of a conic bundle.
\end{remark}
\begin{lemma}\label{lemma: conic bundles lemma}
Let $S$ be a smooth projective surface over a perfect field $K$ with $\HH^{1}(S,\mathcal{O}_{S})=0$ and a curve $C\subset S$ of arithmetic genus 0. Suppose $-K_{S}\cdot C=2$, then the complete linear system $\mid C \mid$ is a pencil. Moreover, if $\mid C \mid$ has no fixed component then $\mid C \mid$ induces a conic bundle.
\end{lemma}
\begin{proof}
It is sufficient to prove the statement over $\bar{K}$. Consider the exact sequence \[
0\rightarrow \mathcal{O}_{S}\rightarrow \mathcal{O}_{S}(C)\rightarrow \mathcal{O}_{C}(C)\rightarrow 0.
\] Applying Zariski cohomology and using $\HH^{1}(S,\mathcal{O}_{S})=0$ we get an exact sequence \[
0\rightarrow \HH^{0}(S,\mathcal{O}_{S})\rightarrow \HH^{0}(S,\mathcal{O}_{S}(C))\rightarrow \HH^{0}(S,\mathcal{O}_{C}(C))\rightarrow 0.
\] As $C^{2}=0$ we can use adjunction to show $\dim\HH^{0}(S,\mathcal{O}_{C}(C))=1$ and as $S$ is projective $\dim\HH^{0}(S,\mathcal{O}_{S})=1$. Applying a dimension argument for vector spaces it is easy to see $\dim \HH^{0}(S,\mathcal{O}_{S}(C))=2$. We can conclude $\dim \mid C \mid =1$ and that it is a pencil. If $\mid C\mid$ has no fixed component, using $C^{2}=0$ and the fact that $\mid C \mid$ is a pencil gives that $\mid C \mid$ is base point free. Using Bertini \cite[Thm 8.18]{H77} over $\bar{K}$, a generic member is smooth and irreducible, hence integral. We can now deduce the generic fibre of the morphism $S\rightarrow \mathbb{P}^{1}$ is a geometrically integral smooth curve of arithmetic genus 0.
\end{proof}
\begin{lemma}\label{lemma: conic bundle for dP2} Let $X$ be a singular del Pezzo surface over a field of odd characteristic with one double rational point singularity, then $X'$ has a conic bundle structure $\phi:X'\rightarrow \mathbb{P}^{1}$.
\end{lemma}
\begin{proof}
Denote by $x$ the singular point on $X$ and $E:=\pi^{-1}(x)$. Consider the curve linear system $\mid -K_{X'}-E\mid $. It easy to check that \[
(-K_{X'}-E)^2=0 \text{ and } -K_{X'}\cdot (-K_{X'}-E)=2.
\]  We can check via Riemann-Roch for surfaces that the arithmetic genus of $D\in \mid -K_{X'}-E\mid$ is 0. By Lemma \ref{lemma: conic bundles lemma}, it is now sufficient to prove that the complete linear system $\mid-K_{X'}-E\mid$ has no fixed component. All $D\in \mid-K_{X'}-E\mid$ correspond to the strict transform of curves on $X$ through the singular point $x$. Using the anticanonical morphism $f:X\rightarrow \mathbb{P}^2$, we see that curves through the singular point on $X$, get mapped to lines through the singular point on the branch locus. As this family of lines forms a pencil, it is clear there is no fixed component, showing that $\mid -K_{X'}-E\mid$, has no fixed competent.
\end{proof}

\subsection{Remaining cases}
Running algorithm \ref{algorithm: determine traces on singular del pezzo} through all possible cases of singular del Pezzo surfaces of degree 2 we see that the cases for which we don't know that $X^{\text{smooth}}(\mathbb{F}_{q})\neq \emptyset$ are as follows.
\begin{listofremainingcase}~
\begin{enumerate}
\item $A_{1}$ over $\mathbb{F}_{2},\mathbb{F}_{3}$ and $\mathbb{F}_{5}$,
\item $A_{2}$ over $\mathbb{F}_{2}$ and $\mathbb{F}_{4}$,
\item $[3A_{1}]',[3A_{1}]''$,$D_{4},A_{3}+2A_{1},2A_{2}+A_{1}$ over $\mathbb{F}_{2}$,
\item $A_{3},[4A_{1}]',[4A_{1}]''$ over $\mathbb{F}_{3}$,
\item $A_{4}$ over $\mathbb{F}_{4}$. 
\end{enumerate}
\end{listofremainingcase}
The list describes the singularity type over $\mathbb{F}_{q}$ and in each case this is the same as the geometric singularity type i.e.\ the singularity type over $\bar{\mathbb{F}}_{q}$. To establish Theorem \ref{thm: main thm} we deal with the cases in odd characteristic. Moreover, we give some examples where $X^{\text{smooth}}(\mathbb{F}_{q})=\emptyset$ in the remaining cases to establish Theorem \ref{thm: main thm 2}.
\begin{proposition}\label{prop: A_1 sing for dP2} Let $X$ be a singular del Pezzo of degree 2 over $\mathbb{F}_{q}$ with $q$ odd and $\bar{X}$ having one singular point which is an $A_{1}$-singularity. Then $X$ has a smooth rational point.
\end{proposition}
\begin{proof}
Denote by $p\in X$ the singular point and let $\pi:X'\rightarrow X$ be the minimal desingularisation of $X$ with $E$ the exceptional divisor lying above $p$. The linear system $|-K_{X'}-E|$ gives a conic bundle structure on $X'$ i.e. a morphism $\phi:X'\rightarrow \mathbb{P}^1$ by Lemma \ref{lemma: conic bundles lemma}. As \[
(-K_{X'}-E)\cdot E = 2
\] we have that $E$ meets each fibre twice (counted with multiplicity). If there exists a smooth fibre of $\phi$ defined over $\mathbb{F}_{q}$, then this fibre has a rational point not lying on $E$, then the image of this rational point under $\pi$ is smooth. Suppose every fibre above $\mathbb{P}^{1}(\mathbb{F}_{q})$ is singular. If $E$ intersects a singular fibre away from its singular point then we are done as this fibre is split over $\mathbb{F}_{q}$, namely has a rational point lying away from $E$. Consider the restriction $\phi|_{E}:E\rightarrow \mathbb{P}^1$, this is a degree 2 morphism and by Riemann-Hurwitz we have that the degree of the ramification divisor of $\phi$ is 2 i.e.\ $E$ cannot intersect each singular fibres at its singular points, hence there is a split singular fibre over $\mathbb{F}_{q}$.
\end{proof}

\begin{proposition}\label{prop: A_3 sing for dP2} Let $X$ be a del Pezzo of degree 2 with an $A_{3}$-singularity over $\mathbb{F}_{q}$ with $q$ odd. Then $X$ has a smooth point.
\end{proposition}
\begin{proof}
Let $\pi:X'\rightarrow X$ be the minimal desingularisation of $X$ and denote by $E_{1},E_{2},E_{3}$ the geometrically irreducible curves in the exceptional divisor of $\pi$. Let $C:=-K_{X'}-E_{1}-E_{2}-E_{3}$ and we have a conic bundle structure $\phi:X'\rightarrow \mathbb{P}^1$ given by the linear system $|C|$ by Lemma \ref{lemma: conic bundles lemma}. Without loss of generality we can assume $E_{1}\cdot E_{2}=E_{2}\cdot E_{3}=1$ and $E_{1}\cdot E_{3} = 0$. Then \[
E_{1}\cdot C = 1, E_{2}\cdot C = 0, E_{3}\cdot C=1.
\] Suppose there exists a smooth fibre of $\phi$ defined over $\mathbb{F}_{q}$, then it is automatic that $X$ has a smooth point as there will a rational point on $X'$ not lying on $E_{i}$ for $i=1,2,3$. Suppose all the fibres of $\phi$ defined over $\mathbb{F}_{q}$ are singular, then $E_{2}$ is a component of at most one singular fibre, hence there is at least one singular fibre where $E_{2}$ is not a component. We denote this fibre by $F$. As $E_{1}$ and $E_{2}$ do not intersect $F$ at its singular point, there is a $\mathbb{F}_{q}$-point of $X'$ not lying on any $(-2)$-curves. The statement then follows from considering the image of this point under $\pi$.
\end{proof}

\begin{lemma}[{\cite[p. 452]{D12}}] Let $X$ be a singular del Pezzo surface of degree 2 over an algebraically closed field $K$ of $\charr(K)\neq 2$. Suppose $X$ has 4 singular points of type $A_{1}$ defined over $K$. Then the ramification curve of the anticanonical morphism is either 
\begin{enumerate}
	\item irreducible cubic + line
	\item smooth conic + smooth conic
\end{enumerate}
\end{lemma}
\begin{lemma}\label{lemma: Galois conj conics for dP2}
Let $X$ be the surface \[
X:w^2=C_{1}C_{2}\subset \mathbb{P}(1,1,1,2)
\] where $C_{1}$ and $C_{2}$ are smooth conics which are Galois conjugates intersecting at $4$ points over a finite field $\mathbb{F}_{q}$ of odd characteristic. Then $X$ has a smooth rational point.
\end{lemma}
\begin{proof}
The above surface defines a singular del Pezzo surface with a singularity type of $4A_{1}$. By Corollary \ref{cor: number of singular points not congruent to one mod q then have smooth point} we have $X$ has a smooth rational point if $q\neq 3$, hence we can assume $q=3$. We can compute all possible smooth conics over $\mathbb{F}_{q^2} = \mathbb{F}_{q}[x]/(x^2-x-1)$. As we specify 4 points in general position for which the conics have to pass through there are $\#\mathbb{P}^{1}(\mathbb{F}_{q^2})=q^2+1$ choices for $C_{1}$. By assumption $C_{1}$ is not defined over $\mathbb{F}_{q}$, hence we have $q^2-q$ choices for $C_{1}$. Moreover, the $q^2-q$ come in Galois conjugate pairs so we can further reduce our choices for $C_{1}$ to $(q^2-q)/2$. It is now easy to iterate all possible choices for $C_{1}$ and $C_{2}$ and see that there always exists $[x:y:z]\in \mathbb{P}^2(\mathbb{F}_{q})$ such that $C_{2}(x,y,z)=C_{2}(x,y,z)\neq 0$. We can then deduce that $X$ has a smooth rational point as all singular points $[x:y:z:w]$ have the property $C_{2}(x,y,z)=C_{2}(x,y,z)=0$.
\end{proof}
\begin{proposition}\label{prop: 4 singularties}
Let $X$ be a singular del Pezzo surface of degree 2 over $\mathbb{F}_{q}$ with $q$ odd. Suppose $X$ has $4A_{1}$ singularities defined over $\mathbb{F}_{q}$ and $\bar{X}$ also has a singularity type of $4A_{1}$, then $X$ has a smooth point.
\end{proposition}
\begin{proof}
By Corollary \ref{cor: number of singular points not congruent to one mod q then have smooth point} if $q\neq 3$ then $X$ has a smooth rational point, hence we can assume $q=3$. Denote by $\pi:X\rightarrow \mathbb{P}^2$ the anticanonical morphism associated to $X$. If the branch locus $B$ of $\pi$ is geometrically $B=C\cup L$ where $C$ is a cubic and $L$ is a line, by Bezout's Theorem we must have the 3 singular points on $B$ lying on $L$. However, as $L(\mathbb{F}_{q})= 4$ we can then deduce $X$ has a smooth rational point. We are now left with the situation where the branch locus is geometrically a union of two smooth conics, $B=C_{1}\cup C_{2}$. If $C_{1}$ and $C_{2}$ are both defined over $\mathbb{F}_{q}$ there would be a rational point on $C_{1}$ and $C_{2}$ which does not correspond to the image of a singular point on $X$ i.e.\ $X$ has a smooth rational point in this case also. The only case left is where $C_{1}$ and $C_{2}$ are Galois conjugates, however this is dealt with in Lemma \ref{lemma: Galois conj conics for dP2}.
\end{proof}
\subsection{Proof of Theorem \ref{thm: main thm} and Corollary \ref{cor: Main corollary}}\label{subsec: proof of Main Theorem and Corollary }
\begin{proof}[Proof of Theorem \ref{thm: main thm}]
For degree $d\geq 3$, the Theorem follows from Corollaries \ref{cor: degree 6, 7 and 8}, \ref{cor: degree 5}, \ref{cor: degree 4} and \ref{cor: dp3}. Running algorithm \ref{algorithm: determine traces on singular del pezzo} through all possible cases of singular del Pezzo surfaces of degree 2 we see that the only cases where we can have all rational points being singular away from $q=2$ and $4$ is $A_{3}, [4A_{1}]', [4A_{1}]''$. However, these cases are dealt with in Propositions \ref{prop: A_1 sing for dP2}, \ref{prop: A_3 sing for dP2} and \ref{prop: 4 singularties}.
\end{proof}
\begin{proof}[Proof of Corollary \ref{cor: Main corollary}]
Theorem \ref{thm: main thm} shows that a singular del Pezzo surface of degree $d\geq 3$ over a finite field, always has a smooth rational point. To show the corollary, it is sufficient now to apply \cite[Thm 9.1]{CT88} for $d\geq 5$, \cite[Lemma 7.19a]{CT88} for $d=4$.
\end{proof}

\subsection{Counterexamples}\label{subsec: counterexamples}
Before we give some examples of singular del Pezzo surfaces of degree 2 with only singular rational points, we explain how we obtained these examples.
\begin{algorithm} Let $X$ be a singular del Pezzo surface of degree 2 over a finite field $k$. Denote by $X'\rightarrow X$ the minimal desingularisation of $X$ and suppose the number of $(-2)$-curves on $\bar{X}'$ is less than 7. Then by Corollary \ref{cor: separable morphism for dP2s} the anticanonical morphism $f:X\rightarrow \mathbb{P}^2$ is separable. Denote by $R$ the ramification divisor of $f$ then \[
f\mid_{X\backslash R}:X\backslash R\rightarrow \mathbb{P}^2\backslash f(R)
\] is étale, hence all singular points on $X$ lie on $R$. In the case that the singular subscheme of $X$, denoted by $X^{\text{sing}}$ has $\#X^{\text{sing}}(k)=\#X(k)$ we must have the following two conditions holding \[(X\backslash R)(k)=\emptyset \text{ and } \#X(k)=\#R(k).
\] As a rational point on the ramification curve maps to a rational point on the branch curve $B$ of $f$ we have $\#B(k)=\#R(k)$. Moreover, as we are working over a finite field we can search for every possible such $B$ with $\#B(k)=\#X(k)$.

\end{algorithm}
\begin{lemma}[{$A_{1}$-singularity}]\label{lemma: A1}
Consider the surface \[
X:w^2+w(y^2+yz+z^2)+(x^2yz+xyz^2+y^4+y^2z^2+z^4)\subset \mathbb{P}(1,1,1,2)
\] over $\mathbb{F}_{2}$. Then $X$ has an $A_{1}$-singularity type and no smooth rational point.
\end{lemma}
\begin{proof}
It easy to check that $X(\mathbb{F}_{2})=\{[1:0:0:0]\}$ and that this $\mathbb{F}_{2}$-point is singular. Using Magma we see that when we blow up this point the exceptional divisor is a single irreducible curve $E$ of self intersection $-2$, hence $E$ is isomorphic to $\mathbb{P}^1$. We can then deduce that $X$ has an $A_{1}$ singularity type and with no smooth rational point.
\end{proof}

\begin{lemma}[{$3A_{1}$-singularity}]\label{lemma: 3A1}
Consider the surface \[
X:w^2+w(x^2)+(x^4+x^2yz+xy^2z+xyz^2+y^2z^2) \subset \mathbb{P}(1,1,1,2)
\] over $\mathbb{F}_{2}$. Then $X$ has an $3A_{1}$-singularity type and no smooth rational point.
\end{lemma}
\begin{proof}
Note $X(\mathbb{F}_{2})=\{[0:1:0:0],[0:0:1:0],[0:1:1:1]\}$. It is easy to check that the $\mathbb{F}_{2}$-points are singular and using Magma we see that when we resolve these singularities. The exceptional divisor contains three disjoint irreducible curves $E_{1},E_{2},E_{3}$ each of self intersection $-2$, hence $E_{i}$ is isomorphic to $\mathbb{P}^1$ for $i=1,2,3$. We can then deduce that $X$ has three singular points each of type $A_{1}$.
\end{proof}
\begin{lemma}[{$D_{4}$-singularity}]\label{lemma: D4}
Consider the surface \[
X:w^2+w(y^2 + yz + z^2)+(xy^2z + xyz^2 + y^4 + y^2z^2 + z^4) \subset \mathbb{P}(1,1,1,2)
\] over $\mathbb{F}_{2}$. Then $X$ has an $D_{4}$-singularity type and no smooth rational point.
\end{lemma}
\begin{proof}
We can check that $X(\mathbb{F}_{2})=\{[1:0:0:0]\}$. It is easy to show that the $\mathbb{F}_{2}$-point is singular and using Magma we see that when we resolve these singularities the exceptional divisor contains one connected component with the following intersection graph 
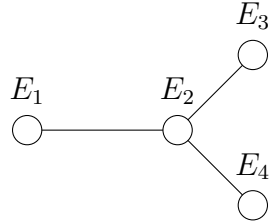
\begin{center}
\begin{tikzpicture}
\draw (0,0) -- (2,0);
\draw (2,0) -- (3,-1);
\draw (2,0) -- (3,1);
\node[draw, circle,fill=white,label = $E_{1}$]at (0,0){};
\node[draw, circle,fill=white,label = $E_{2}$]at (2,0){};
\node[draw, circle,fill=white,label = $E_{3}$]at (3,1){};
\node[draw, circle,fill=white,label = $E_{4}$]at (3,-1){};
\end{tikzpicture}
\end{center} and each $E_{i}$ has self intersection $-2$ for $i=1,2,3,4$. Hence, $X$ has a $D_{4}$-singularity type.
\end{proof}

\begin{proof}[Proof of Theorem \ref{thm: main thm 2}]
The statement follows from Lemmas \ref{lemma: A1},\ref{lemma: 3A1} and \ref{lemma: D4}.
\end{proof}

\bibliographystyle{amsalpha}{}
\bibliography{referencesdelPezzo}
\end{document}